\documentclass[12pt]{amsart}
\usepackage{amssymb,latexsym,amsmath,amscd,mathrsfs,yfonts,hyperref}
\usepackage{fullpage}

\input xy
\xyoption{all}

\newcommand{\RR}{\mathbb R}

\newcommand{\QQ}{\mathbb Q}
\newcommand{\CC}{\mathbb C}

\newcommand{\NN}{\mathbb N}

\newcommand{\dlim}{{\varinjlim}}
\newcommand{\ilim}{{\varprojlim}}

\def\loc{\textup{loc}}

\def\C{\textup{C}}

\def\ev{\textup{ev}}

\def\id{\mathrm{id}}

\def\X{\mathcal X}

\def\HL{\textup{HL}}

\def\Ban{\mathtt{Ban}}

\def\K{\textup{K}}

\def\diss{\textup{diss}}

\def\prot{\hat{\otimes}}

\def\H{\textup{H}}

\def\pct{\textup{Cpt}}

\def\1{\bf{1}}
\def\sC{\text{$\sigma$-$C^*$}}

\def\an{\textup{an}}

\def\grot{\hat{\otimes}_\pi}

\def\cT{\mathcal T}

\def\cB{\mathcal{B}}

\newcommand{\map}{\rightarrow}

\def\cT{\mathcal T}

\newcommand{\beq}{\begin{eqnarray}}
\newcommand{\beqn}{\begin{eqnarray*}}
\newcommand{\eeq}{\end{eqnarray}}
\newcommand{\eeqn}{\end{eqnarray*}}

\newtheorem{thm}{Theorem}

\newtheorem{lem}[thm]{Lemma}
\newtheorem{prop}[thm]{Proposition}
\newtheorem{cor}[thm]{Corollary}

\newtheorem{rem}[thm]{Remark}

\begin{document}

\title{Continuous homotopy invariance of bivariant local cyclic homology for $\sC$-algebras}
\author{Snigdhayan Mahanta}
\email{s.mahanta@uni-muenster.de}
\address{Mathematical Institute,
University of Muenster,
Einsteinstrasse 62,
48149 Muenster, Germany}
\thanks{This research was supported under Australian Research Council's Discovery Projects funding scheme (project number DP0878184) and ERC through AdG 267079.}

\begin{abstract}
We establish the continuous homotopy invariance of bivariant local cyclic homology on the category of all $\sC$-algebras. The argument relies vitally on an isomorphism between the smooth and the continuous cylinder constructions using a technical criterion due to Meyer. As a consequence we compute the local cyclic homology of the infinite sphere.
\end{abstract}

\maketitle
It is a widely regarded fact that $\K$-theory and (cyclic) homology theory are two pillars of noncommutative geometry via topological algebras. The study of various (co)homology theories has a rich and successful history (see, e.g., \cite{ConBook,Lod,Helemskii,MeyCycHom}). One variant of them is the $2$-periodic bivariant local cyclic homology that was used effectively to prove the Kadison--Kaplansky conjecture for torsion free word-hyperbolic groups in the sense of Gromov \cite{PusKKConj}. In this article we prove the continuous homotopy invariance of bivariant local cyclic homology on the category of all $\sC$-algebras. Let us briefly comment on the scope and applicability of $\sC$-algebras. Since the category of $\sC$-algebras contains all $C^*$-algebras, the former category can treat all locally compact Hausdorff (noncommutative) spaces. Moreover, it can treat non-locally compact countably compactly generated Hausdorff spaces and noncommutative versions thereof (see Section 5 of \cite{NCP1}). {\em Rather surprisingly, $\C(\QQ)$ is not a $\sC$-algebra (see Example 5.8. of ibid.), although $\QQ$ is a countable set with a metrizable topology.}

To any Fr{\'e}chet algebra one can apply the {\em precompact bornology} functor $\pct(-)$, which converts it into a complete bornological algebra. Then one can apply the dissection functor $\diss(-)$ to produce an ind-Banach algebra. Henceforth let us set $P(-):=\diss\circ\pct(-)$ to avoid notational clutter. Now the $\X$-complex formalism may be deployed to the (bornologically completed) analytic tensor algebra $\cT_\an(-)$ to define the bivariant local cyclic homology of two Fr{\'e}chet algebras $A,B$ as $$\HL_*(A,B)=\H_*^\loc(\X(\cT_\an(P(A))),\X(\cT_\an(P(B)))).$$ For the details concerning the rest we refer the readers to \cite{MeyCycHom}. This is how we are going to define the bivariant local cyclic homology on the category of all Fr{\'e}chet algebras following Meyer. At first sight this definition might differ from the original one of bivariant local cyclic homology by Puschnigg \cite{Puschnigg,PuschniggHL}; however, it produces a theory that is naturally isomorphic to the original one (see Theorem 5.38 of \cite{MeyCycHom}). Although we are mainly interested in the bivariant local cyclic homology of $\sC$-algebras, we need to look at the theory for the larger class of all Fr{\'e}chet algebras. This is because in the arguments below we need to consider both smooth and continuous cylinder constructions and the smooth cylinder construction takes us beyond the category of $\sC$-algebras (but keeps us within that of Fr{\'e}chet algebras).

Let us recall that a {\em locally multiplicative ind-Banach} algebra is an ind-Banach algebra, which is an inductive system of Banach algebras; whereas an arbitrary ind-Banach algebra is merely a monoid object in the symmetric monoidal category of ind-Banach spaces. A Fr{\'e}chet algebra $A$ is called {\em locally multiplicative} if its associated ind-Banach algebra $P(A)$ is locally multiplicative. It is known that bivariant local cyclic homology satisfies smooth homotopy invariance on the category of all ind-Banach algebras. This was a designing criterion in \cite{Puschnigg,PuschniggHL}; in the generality of ind-Banach algebras (as discussed above) this result is proved in Theorem 5.45 of \cite{MeyCycHom}. It is further known that bivariant local cyclic homology satisfies continuous homotopy invariance on the category of all locally multiplicative Fr{\'e}chet algebras (see Corollary 6.29 of ibid.). From the characterization of locally multiplicative Fr{\'e}chet algebras (see Theorem 3.34. of ibid.) one knows that such algebras cannot have elements with unbounded spectral radii. A $\sC$-algebra is also a Fr{\'e}chet algebra. However, it contains many elements with unbounded spectral radii, unless it is actually a $C^*$-algebra. Therefore, such a $\sC$-algebra cannot be locally multiplicative and hence the existing result on the continuous homotopy invariance is not applicable to the category of all $\sC$-algebras. 

\medskip

\noindent
{\bf Terminology and convention:} {\em In the sequel, by a Fr{\'e}chet algebra (resp. a $\sC$-algebra) we mean a complete Hausdorff locally convex $\CC$-algebra that can written as a countable inverse limit of Banach algebras (resp. $C^*$-algebras). Such Fr{\'e}chet algebras are called locally multiplicatively convex in the literature (see, for instance, \cite{MichaelLMC}); the readers should not confuse them with the locally multiplicative Fr{\'e}chet algebras described above. All bivariant local cyclic homological constructions are applied to Fr{\'e}chet algebras after applying the functor $P(-)$, which is suppressed from the notation for brevity.}

\medskip

Let $A$ be any Fr{\'e}chet algebra. A bounded (= precompact) subset $T\subset \pct(A)$ is called {\em power bounded} if $T^\infty :=\cup_{n=1}^\infty T^n$ is bounded and the {\em spectral radius} of $T$, denoted by $\varrho(T)$, is the infimum of the numbers $r\in\RR_{>0}$ for which $r^{-1}T$ is power bounded. If no such $r$ exists, then the spectral radius of $T$ is set to $\infty$. Let $(\Ban,\grot)$ denote the symmetric monoidal category of all Banach spaces with Grothedieck's projective tensor product $\grot$. Note that a partial algebra in $\Ban$ is a triple $(X,\mu,R)$, where $X,R$ are objects of $\Ban$ and $\mu:R\map X$ is a morphism therein; an injective map $R\map X\grot X$ in $\Ban$ is also a part of the data that is conveniently suppressed from the notation. Let $f: X\map B$ be a linear map, where $B$ is an ind-Banach algebra. The {\em curvature} of $f$, denoted by $\omega_f: R\map B$, is the difference of the two maps

\beqn
R\overset{\mu}{\map} X\overset{f}{\map} B \quad\text{ and }\quad R\map X\grot X\overset{f\grot f}{\longrightarrow} B\grot B\overset{\textup{mult}}{\map} B.
\eeqn

For a fixed bounded disk $S\subset R$, the linear map $f:X\map P(A):=\diss\circ\pct(A)$ is said to have {\em small curvature} with respect to $S$, if the restriction of $\omega_f$ to $S$ has {\em spectral radius} strictly less than $1$, where the {\em spectral radius} of a bounded map $g:S\map P(A)$ is simply the spectral radius of the bounded subset $g(S)$ inside $\pct(A)$. The previous assertion uses the fact that $f$ is a linear map into an ind-Banach algebra of the form $P(A)$, where $A$ is a Fr{\'e}chet algebra. Let $M(S;X,A)$ be the set of linear maps $X\map P(A)$ with small curvature with respect to $S$ and $H(S;X,A)$ be the set of smooth homotopy classes of such maps, where the homotopies must themselves be elements of $M(S;X,\C^\infty([0,1],A))$, i.e., have small curvature with respect to $S$. A homomorphism $f: A\map B$ of ind-Banach algebras is called an {\em approximate smooth local homotopy equivalence} if $\cT_\an(f):\cT_\an(A)\map\cT_\an(B)$ is a smooth local homotopy equivalence (see section 6.1.2. of \cite{MeyCycHom}). Let us recall

\begin{thm}[Meyer, Theorem 6.19. \cite{MeyCycHom}] \label{Meyer}
An algebra homomorphism $f:A\map B$ between two ind-Banach algebras is an approximate smooth local homotopy equivalence if and only if it induces isomorphisms $H(S;X,A)\map H(S;X,B)$ for all partial algebras $(X,\mu,R)$ with a fixed unit ball $S\subset R$.
\end{thm}

In the sequel we denote by $\ev^\infty_i:A^\infty[0,1] :=\C^\infty([0,1],A)\map A$ and $\ev_i: A[0,1] := \C([0,1],A)\map A$, $i=0,1$, the smooth and continuous cylinder constructions respectively with their corresponding evaluation maps. Both cylinder constructions are nicely compatible with the functor $P(-)$, provided the input is a Fr{\'e}chet algebra. Due to the nuclearity of $\C^\infty([0,1])$ as a Fr{\'e}chet space, one has $A^\infty[0,1]\cong \C^\infty([0,1])\grot A$. Observe that $A[0,1]= \C([0,1])\prot_{\sC} A$ (maximal or minimal $\sC$-tensor product) and $A^\infty[0,1]$ are both Fr{\'e}chet algebras themselves (see, for instance, Sections 1 and 2 of \cite{PhiFreK}). Recall that an algebra homomorphism $f$ between ind-Banach algebras is called an {\em $\HL$-equivalence} if $\X(\cT_\an(f))$ is a local homotopy equivalence.

\begin{lem} \label{SR}
For any $\sC$-algebra $A$, the canonical Fr{\'e}chet algebra homomorphism between the cylinder constructions $\iota: A^\infty[0,1]\map A[0,1]$ preserves spectral radii of bounded subsets as a bounded map $\iota:\pct(A^\infty[0,1])\map\pct(A[0,1])$, i.e., for any bounded subset $T\subset \pct(A^\infty[0,1])$ one has $\varrho(T)=\varrho(\iota(T))$ (both could be infinite).
\end{lem}

\begin{proof}
The continuous homomorphism $\iota$ gives rise to a bounded homomorphism between complete bornological algebras $\iota:\pct(A^\infty[0,1])\map\pct(A[0,1])$, from which it automatically follows that $\varrho(\iota(T))\leqslant \varrho(T)$ for any bounded subset $T\subset \pct(A^\infty[0,1])$. Let us set $T_r = r^{-1}T$. Suppose that $T_r^\infty$ viewed as a subset of $\pct(A[0,1])$ via $\iota$ is bounded, i.e., $\overline{T_r^\infty}$ is compact. We are going to show that the closure of ${T_r^\infty}$ inside $\pct(A^\infty[0,1])$ is also compact, i.e., $T_r^\infty$ viewed as a subset of $\pct(A^\infty[0,1])$ is bounded. For simplicity we drop $\pct(-)$ from the corresponding bornological algebras in the next paragraph.

Due to the metrizability of the topology of $A^\infty[0,1]$ it suffices to check sequential compactness. Given any sequence $\{a_n\}$ of elements in the closure of ${T_r^\infty}$ inside $A^\infty[0,1]$, we view them via $\iota$ as a sequence in $\overline{T_r^\infty}\subset A[0,1]$. Thanks to the compactness of $\overline{T_r^\infty}\subset A[0,1]$ we can extract a convergent subsequence $\{b_n\}\map b\in A[0,1]$. Let $d$ be the Fr{\'e}chet metric built out of the countable family $\{p_k\}_{k\in\NN}$ of submultiplicative seminorms on $A^\infty[0,1]$ as $$d(x,y)=\sum_{k=0}^\infty \frac{2^{-k}p_k(x-y)}{1+p_k(x-y)},\quad\quad \forall x,y\in A^\infty[0,1]$$ and, similarly, let $d'$ be the Fr{\'e}chet metric on $A[0,1]$. For any $\epsilon$-ball $B(d,\epsilon)\subset A^\infty[0,1]$ there is a $\delta$-ball $B(d',\delta)\subset A[0,1]$, such that $\delta<\epsilon$ and $B(d',\delta)\cap A^\infty[0,1]\subset B(d,\epsilon)$. Consequently, for any $\epsilon>0$ we may choose a $B(d',\delta)$ around $b$ with $\delta<\epsilon$, such that there is an $N\in\NN$ with $b_n\in B(d',\delta)\cap A^\infty[0,1]$ for all $n> N$. It follows that $d(b_n,b_m)<\epsilon$ for all $n,m>N$ proving the that sequence $\{b_n\}$ is Cauchy inside $A^\infty[0,1]$. By the completeness of the closure of $T_r^\infty$ inside $A^\infty[0,1]$ we conclude that it is convergent.

This analysis yields that the spectral radius of $T\subset\pct(A^\infty[0,1])$ is less than or equal to that of $\iota(T)\subset\pct(A[0,1])$ from which the assertion follows.
\end{proof}

\begin{prop}
For any $\sC$-algebra $A$, the canonical Fr{\'e}chet algebra homomorphism between the cylinder constructions $\iota: A^\infty[0,1]\map A[0,1]$ is an $\HL$-equivalence.
\end{prop}

\begin{proof}
By Theorem 6.11. of \cite{MeyCycHom} it suffices to show that $\iota$ is an approximate smooth local homotopy equivalence. The above Theorem \ref{Meyer} gives us a criterion to ascertain that, which we presently verify.

Let $(X,\mu,R)$ be any partial algebra in $\Ban$ and $S\subset R$ be a fixed bounded disk. For the surjectivity of the induced map $\iota: H(S;X,A^\infty[0,1])\map H(S;X,A[0,1])$ we need to find for any $h\in M(S;X,A[0,1])$ an $h'\in M(S;X,A^\infty[0,1])$, such that $\iota h'$ is homotopic to $h$ via a smooth homotopy of small curvature. Let us choose a sequence of smooth kernels $K_n(x,y)$ converging to $\delta(x-y)$. Using this sequence we produce via kernel smoothing a sequence of continuous linear maps (see, for instance, chapter 40 of \cite{Treves})
\beqn
\sigma_n:\C([0,1])&\map&\C^\infty([0,1])\subset\C([0,1])\\
f(x) &\mapsto& [x\mapsto \int K_n(x,y)f(y)dy].
\eeqn One extends them to linear maps $\sigma_n=\sigma_n\otimes\id:\C([0,1])\otimes A\map\C^\infty([0,1])\otimes A$, where $\otimes$ denotes the algebraic tensor product. This sequence can be further extended to a sequence of continuous linear maps $\sigma_n: A[0,1] \map A^\infty[0,1]$ by density arguments, such that $\iota \sigma_n: A[0,1]\map A[0,1]$ uniformly converges towards $\id$. This implies that $\iota\sigma_n h$ uniformly converges towards $h$, whence $\omega_{\iota\sigma_n h}$ uniformly converges towards $\omega_h$ (for a judicious choice of the approximation scheme via smooth kernels). The spectral radius of $\omega_{\sigma_n h}: R\map A^\infty[0,1]$ (resp. $\omega_{\iota\sigma_n h}: R\map A[0,1]$) restricted to $S$ is that of the image $\omega_{\sigma_n h}(S)$ (resp. $\omega_{\iota\sigma_n h}(S)$) inside the bornological algebra $\pct(A^\infty[0,1])$ (resp. $\pct(A[0,1])$). In the next paragraph we are going to show that $\varrho(\omega_{\sigma_n h}(S))$ is strictly less than $1$ for sufficiently large $n$, whence such a $\sigma_n h$ is a `potential candidate' for the desired $h'\in M(S;X,A^\infty[0,1])$.

The $\sC$-algebra $A[0,1]$ has a countable family of $C^*$-seminorms $\{p_n\}_{n\in\NN}$. For any $a\in A[0,1]$, set $\|a\|_\infty = \sup_n \{p_n(a)\}$ and define $b(A[0,1]):=\{a\in A\,|\, \|a\|_\infty <\infty\}$. Then $b(A[0,1])$ is a $C^*$-algebra with $\|\cdot \|_\infty$-norm and the inclusion $b(A[0,1])\hookrightarrow A[0,1]$ is continuous with dense image (see Proposition 1.11. of \cite{NCP1}). It is clear that $\varrho(\omega_{\iota\sigma_n h}(S)\cap b(A[0,1]))$ converges to $\varrho(\omega_h(S)\cap b(A[0,1]))$. By density arguments it follows that the spectral radius of $\omega_{\iota\sigma_n h}(S)$ converges to that of $\omega_h(S)$, which is strictly less than $1$ by hypothesis. From the previous Lemma \ref{SR}, we know that $\varrho(\omega_{\sigma_n h}(S))=\varrho(\omega_{\iota\sigma_n h}(S))$. Since the sequence $\{\varrho(\omega_{\iota\sigma_n h}(S))\}$ converges to a number strictly less than $1$, we conclude that $\varrho(\omega_{\sigma_n h}(S))$ is strictly less than $1$ for sufficiently large $n$.

There is an affine smooth homotopy $H_n: X\map (A[0,1])^\infty[0,1]$, where $H_n(t)=(1-t)h + t\iota\sigma_n h$, connecting the two linear maps $h$ and $\iota\sigma_n h$. In the following paragraph we show that for sufficiently large $n$ this  affine homotopy $H_n$ also has small curvature, which is equivalent to the assertion that the spectral radius of $\omega_{H_n}(S)$ is strictly less than $1$ for sufficiently large $n$. This will show that for a sufficiently large $n$, the element $\sigma_n h\in M(S;X,A^\infty[0,1])$ is a legitimate candidate for the desired $h'\in M(S;X,A^\infty[0,1])$, whence the induced map $\iota: H(S;X,A^\infty[0,1])\map H(S;X,A[0,1])$ must be surjective.

Since $\omega_{H_n}(S)$ is a precompact subset of the Fr{\'e}chet algebra $(A[0,1])^\infty[0,1]$, its closure is compact.  Set $X_n=\omega_{H_n}(S)\subset\C^\infty([0,1],A[0,1])$ and set $X_n(t)$ to be the image of $X_n$ after composition with the evaluation map $\ev^\infty_t:\C^\infty([0,1],A[0,1])\map A[0,1]$ for $t\in[0,1]$. We know that $X_n(0)=\omega_h(S)$ (resp. $X_n(1)=\omega_{\iota\sigma_n h}(S)$), whose spectral radius is $R_0<1$ (resp. $R_1<1$). We write the $\sC$-algebra $A[0,1]$ explicitly as a countable inverse limit $$A[0,1]\cong\ilim_{l\in\NN} B_l,$$ where each $B_l$ is a Banach algebra (actually a $C^*$-algebra), and we denote by $\pi_l:\ilim_l B_l\map B_l$ the canonical projection homomorphisms. Let us introduce the following notation: $$\text{$\pi_l(X_n)= V_l$ and $\pi_l(X_n(t)) = V(t)_l$,}$$ where each $V(t)_l$ is a bounded subset of $B_l$ for every $l\in\NN$ and $\overline{(R_i^{-1}(V(i)_l))^\infty}$ is compact for $i=0,1$. This means that there are natural numbers $N_l$ such that $(R_0^{-1}(V(0)_l))^{N_l}\subset Q_l$, where $Q_l\subset B_l$ denotes the unit ball. By choosing $n$ sufficiently large, we may arrange that $(r^{-1}(V(t)_l))^{N'_l}\subset Q_l$ for all $t\in [0,1]$ and for some $r<1$. This is achieved as follows: due to the Heine--Borel property of Fr{\'e}chet spaces, the compact subset $\overline{(R_0^{-1}(X_n(0)))^\infty} \subset A[0,1]$ is contained in a bounded open convex ball $\cB$ with respect to the standard Fr{\'e}chet metric. We may choose $\cB$ small enough, so that $(r^{-1}(\pi_l(\cB)))^{N'_l}\subset Q_l$ for some $r$ in the range $\textup{max}\{R_0,R_1\}< r<1$. Now using the uniform convergence of $X_n(1)=\omega_{\iota\sigma_n h}(S)$ towards $X_n(0)=\omega_h(S)$ we can ensure that $\overline{(R_1^{-1}(X_n(1)))^\infty}\subset \cB$ for sufficiently large $n$. Observe that
\beq \label{invlimit}
\overline{(r^{-1}(X_n))^\infty}\cong\ilim_l\overline{(r^{-1}(V_l))^\infty}=\ilim_l\overline{(\cup_{p=1}^{N'_l-1}(r^{-1}(V_l))^p)\cup((r^{-1}(V_l))^{N'_l})^\infty}.
\eeq Since $Q_l^2\subset Q_l$, $\overline{Q_l}$ is compact for each $l$ and the inverse limit of compact Hausdorff subsets is again a compact subset, we conclude that $\ilim_l\overline{((r^{-1}(V_l))^{N'_l})^\infty}\subset \ilim_l \overline{Q_l}$ is compact. Since $r^{-1}(V_l)$ is precompact for each $l$, so is $\cup_{p=1}^{N'_l}(r^{-1}(V_l))^p$. It follows that $\ilim_l\overline{\cup_{p=1}^{N'_l}(r^{-1}(V_l))^p}$ is compact, which implies that $\overline{(r^{-1}(X_n))^\infty}$ is compact (see Equation \eqref{invlimit}). Consequently $r^{-1}(\omega_{H_n}(S))=r^{-1}(X_n)$ is power bounded for some $r<1$, whence the spectral radius of $\omega_{H_n}(S)$ is strictly less than $1$.

For injectivity, we need to show that whenever $\iota h_0$ and $\iota h_1$ are homotopic via a homotopy $H\in M(S;X, (A[0,1])^\infty[0,1])$, so are $h_0$ and $h_1$ via a homotopy in $M(S;X,(A^\infty[0,1])^\infty[0,1])$. Adopting a similar strategy as above, we can construct a sequence of smooth homotopies $H'_n \in M(S; X, (A^\infty[0,1])^\infty[0,1])$, such that $\iota H'_n$ converges uniformly to $H$. It follows that $\ev^\infty_0 H'_n$ and $\ev^\infty_1 H'_n$ have the same class in $H(S;X,A^\infty[0,1])$. Once again, for $i=0,1$, there are affine homotopies $G_n: X\map (A^\infty[0,1])^\infty[0,1]$, where $G_n(t)=(1-t)h_i + t\ev^\infty_i H'_n$. It suffices to show that $G_n$ has small curvature for sufficiently large $n$, which would imply that $\ev^\infty_i H'_n$ (for sufficiently large $n$) and $h_i$, for $i=0,1$, have the same class in $H(S;X,A^\infty[0,1])$. Now using arguments as above, we can show that for sufficiently large $n$, the spectral radius of $\omega_{G_n}(S)$ inside $\pct((A^\infty[0,1])^\infty[0,1])$ is strictly less than $1$.
\end{proof}

\begin{thm}
Bivariant local cyclic homology satisfies continuous homotopy invariance on the category of all $\sC$-algebras, i.e., for any two $\sC$-algebras $A,B$ the functors $\HL_*(B,-)$ and $\HL_*(-,B)$ send the evaluation maps $\ev_i:A[0,1]\map A$ for $i=0,1$ to isomorphisms.
\end{thm}

\begin{proof}
Consider the following commutative diagram of Fr{\'e}chet algebras

\beqn
\xymatrix{
& A\\
A^\infty[0,1]\ar[ru]^{\ev^\infty_0}\ar[r]^{\iota}\ar[rd]_{\ev^\infty_1} & A[0,1]\ar[u]_{\ev_0}\ar[d]^{\ev_1}\\
& A,
}
\eeqn where $\ev^\infty_i$ (resp. $\ev_i$) are the evaluation maps from the smooth (resp. continuous) cylinder object. Due to the smooth homotopy invariance of bivariant local cyclic homology, we conclude that $\ev^\infty_i$ are mapped to isomorphisms by $\HL_*(B,-)$ and $\HL_*(-,B)$. The above Proposition shows that $\iota$ in an $\HL$-equivalence, whence it is mapped to an isomorphism by the functors $\HL_*(B,-)$ and $\HL_*(-,B)$. The assertion now follows from the two-out-of-three property of isomorphisms.
\end{proof}

\begin{cor}
We have established that bivariant local cyclic homology satisfies property $(E1)$ of Cuntz \cite{CunGenBivK} on the category of separable $\sC$-algebras.
\end{cor}

\begin{cor}
Owing to the contractibility of the infinite sphere $S^\infty:= \dlim_n S^{2n+1}$, we deduce $$\text{$\HL_0(\CC,\C(S^\infty)) \simeq \CC\quad$ and $\quad\HL_1(\CC,\C(S^\infty))\simeq \{0\}$.   }$$
\end{cor}

\begin{rem}
The incompatiblity between the cylinder constructions and the functor $P(-)$ is one of the serious impediments to extending the above arguments beyond $\sC$-algebras. In fact, the continuous homotopy invariance of bivariant local cyclic homology may very well fail to hold for arbitrary pro $C^*$-algebras.
\end{rem}

 \noindent
 {\bf Acknowledgements:} The author is indebted to R. Meyer for his valuable suggestions and corrections to the first draft and to N. C. Phillips for helpful discussions. The author also gratefully acknowledges the support of Deutsche Forschungsgemeinschaft (SFB 878).

\bibliographystyle{abbrv}
\bibliography{/Users/mahanta/Professional/math/MasterBib/bibliography}

\vspace{5mm}
\noindent

\end{document}